\documentclass[reqno]{amsart}
\usepackage{amsmath}
  \usepackage{paralist}
  \usepackage{graphics} 
  \usepackage{epsfig} 
 \usepackage[colorlinks=true]{hyperref}
\hypersetup{urlcolor=blue, citecolor=red}

\newtheorem{theorem}{Theorem}[section]
\newtheorem{corollary}{Corollary}

\newtheorem{lemma}[theorem]{Lemma}

\theoremstyle{definition}
\newtheorem{definition}[theorem]{Definition}
\newtheorem{remark}{Remark}


 \email{matematika@t-com.me}

\DeclareMathOperator{\esslim}{esslim}
\DeclareMathOperator{\ccup}{\cup} \DeclareMathOperator{\Div }{div}

\def\R{I\!\!R}

\def\N{I\!\!N}

\def\cal{\mathcal}

\def\pa{\partial}
\def\la{\lambda}

\begin{document}

\title[$(\alpha,\beta)$-entropy conditions]
      {New entropy conditions for scalar conservation laws with discontinuous flux}

\author{ D.~ Mitrovi\'c}
\address{ Darko Mitrovic}, \address{University of Montenegro, Faculty of Mathematics, Cetinjski put bb, 81000 Podgorica, Montenegro and}
\address{University of Bergen, Faculty of Mathematics, Johannes Bruns gate 12, 5007 Bergen, Norway}
 \email{  matematika@t-com.me}

\begin{abstract}
We propose new Kruzhkov type entropy conditions for one dimensional
scalar conservation law with a discontinuous flux. We prove
existence and uniqueness of the entropy admissible weak solution to
the corresponding Cauchy problem merely under assumptions on the
flux which provide the maximum principle. In particular, we allow
multiple flux crossings and we do not need any kind of genuine
nonlinearity conditions.
\end{abstract}



\maketitle

In the current contribution, we consider the following problem
\begin{equation}
\label{oslo3}
\begin{cases}
\pa_t u+\pa_x\left(H(x)f(u)+H(-x)g(u) \right)=0, & (t,x)\in  \R^+\times \R\\
u|_{t=0}=u_0(x)\in L^\infty(\R), & x\in \R
\end{cases}
\end{equation}where $u$ is the scalar unknown function;
$u_0$ is a function such that $a\leq u_0 \leq b$, $a,b\in \R$; $H$
is the Heaviside function; and $f,g\in C^1({\bf R})$ are such that
$f(a)=f(b)=g(a)=g(b)=0$.

Problems such as \eqref{oslo3} are non-trivial generalization of
scalar conservation law with smooth flux, and they describe
different physical phenomena (flow in porous media, sedimentation
processes, traffic flow, radar shape-from-shading problems, blood
flow, gas flow in a variable duct...). Therefore, beginning with
eighties (probably from \cite{temple}), problems of type
\eqref{oslo3} are under intensive investigations.

As usual in conservation laws, the Cauchy problem under
consideration in general does not possess classical solution, and it
can have several weak solutions. Since it is not possible to
directly generalize standard theory of entropy admissible solutions
\cite{Kru}, in order to choose a proper weak solution to
\eqref{oslo3} many admissibility conditions were proposed. We
mention minimal jump condition \cite{GR}, minimal variation
condition and $\Gamma$ condition \cite{Die1, Die2}, entropy
conditions \cite{kar3, AG}, vanishing capillary pressure limit
\cite{Kaa}, admissibility conditions via adapted entropies \cite{AP,
kar4} or via conditions at the interface \cite{sid_2, sid, Die3}.

But, in every of the mentioned approaches, in order to prove
existence or uniqueness of a weak solution to the considered
problem, some structural hypothesis on the flux (such as convexity
or genuine nonlinearity) or on the form of the solution (see
\cite{sid_2, sid}) were assumed.

Recently, in \cite{NHM_mit}, we have proved existence and uniqueness
in the multidimensional situation. Still, due to certain technical
obstacles, admissible solutions selected in that paper are rather
special.

Here, we propose admissibility conditions which involve much less
restrictions than in previous works on the subject (excluding
\cite{NHM_mit} where there are no restrictions), and we still can
make many different stable semigroups depending on the physical
situation under considerations.

Since one can find excellent overviews on the subject in many papers
\cite{AKR, sid, Vov, kar4, Die3, Pan_08} which are easily available
via internet (e.g. www.math.ntnu.no/conservation), in this
introduction, we shall restrict our attention on papers \cite{kar3},
\cite{kar1}, and \cite{Pan_08} which are in the closest connection
to our contribution. Later, in Section 2, we shall comment how our
admissibility conditions can be considered as a generalization of
the entropy solution of type $(A,B)$ given in \cite{kar4} (see
Definition \ref{kenneth_2} in the current paper).

In \cite{kar3}, degenerate parabolic equation with discontinuous
flux is considered:
\begin{equation*}
\begin{cases}
\pa_t u+\pa_x\left(H(x)f(u)+H(-x)g(u) \right)=\pa_{xx}A(u), & (t,x)\in (0,T)\times \R\\
u|_{t=0}=u_0(x)\in BV(\R)\cap L^1(\R), & x\in \R,
\end{cases}
\end{equation*}where $A$ is non-decreasing with $A(0)=0$. Assuming that $A\equiv 0$ we obtain the problem of type \eqref{oslo3}.
In order to obtain uniqueness of a weak solution to the problem, the
Kruzhkov type entropy admissibility condition \cite{Kru} is used:

\begin{definition}\label{kar1} \cite{kar3}
Let $u$ be a weak solution to problem (\ref{oslo3}).

We say that $u$ is an entropy admissible weak solution to
(\ref{oslo3}) if the following entropy condition is satisfied for
every fixed $\xi\in {\bf R}$:
\begin{align*}
\nonumber
\pa_t {|u-\xi|}&+\pa_x \Big\{{\rm sgn}(u-\xi)\Big[H(x)({f}(u)-{f}(\xi))+H(-x)({g}(u)-{g}(\xi))\Big]\Big\}\nonumber\\
&\qquad\qquad\qquad\qquad-|{f}({\xi})-{g}({\xi})|\delta(x) \leq 0 \
\ {\rm in} \ \ {\cal D}'(\R^+\times \R).
\end{align*}
\end{definition} Still,
merely such entropy condition was insufficient to prove stability of
the admissible weak solution to the considered problem. Two more
things were necessary.

First, one needs the following technical assumption:

{\bf Crossing condition:} For any states $u, v$ the following
crossing condition must hold:
$$
f(u)-g(u)<0<f(v)-g(v) \Rightarrow u<v.
$$ Geometrically,
the crossing condition requires that either the graph of $f$ and $g$
do not cross, or the graph $g$ lies above the graph of $f$ to the
left of the crossing point (see Figure \ref{1}). The functions $f$
and $g$ appearing in \eqref{oslo3} do not necessarily satisfy the
crossing conditions, but it is possible to transform them so that
the crossing condition is satisfied (see Figure \ref{2} and Figure
\ref{3}).

Next, in \cite{kar1} existence of strong traces at the interface
$x=0$ was necessary. We provide appropriate definition.

\begin{definition}
Let $W : \R\times \R^+\to \R$ be a function that belongs to
$L^\infty(\R\times \R^+)$. By the right and left traces of $W(\cdot,
t)$ at the point $x = 0$ we understand functions $t\mapsto W(0\pm,
t)\in L^\infty_{loc}(\R^+)$ that satisfy for a.e. $t\in \R^+$:
\begin{equation*}
 \esslim\limits_{x\uparrow 0} |W(t,x)-W(t,0+)| = 0, \ \ \esslim\limits_{x\downarrow 0} |W(t,x)-W(t,0-)| = 0
\end{equation*}
\end{definition}

Assuming the crossing condition and the existence of traces, we have
the following theorem:

\begin{theorem} \cite{kar3}
\label{kenn} Assume that weak solutions $u$ and $v$ to \eqref{oslo3}
with the initial conditions $u_0$ and $v_0$, respectively, satisfy
entropy admissibility conditions from Definition \ref{kar1} and
admit left and right strong traces at the interface $x=0$.

Then for any $T, R>0$ there exist constants $C,\bar{R}>0$ such that:
\begin{equation}
\label{new_version} \int_0^T\int_{-R}^R|v(t,x)-u(t,x)| dx dt \leq C
T \int_{-\bar{R}}^{\bar{R}}|v_0(x)-u_0(x)| dx.
\end{equation}
\end{theorem}

\begin{remark}
It is important to notice that Theorem \ref{kenn} remains to hold if
in \eqref{oslo3}, instead of $\pa_t u$, we put $\pa_t
(\alpha(u)H(x)+\beta(u)H(-x))$, for some strictly increasing
bijections $\alpha: [a,b]\to [a',b']$ and $\beta:[a,b]\to [a'',b'']
$, $a',a'',b',b''\in R$. Indeed, since we did not put a function
depending on $t\in \R^+$ under the derivative $\pa_t$, and since
$\alpha$ and $\beta$ are increasing bijections (we can extract all
the information on $u$ knowing only $\beta(u)$ or $\alpha(u)$), we
can safely use results from \cite{kar1} on the equation
$\pa_t(\alpha(u)H(x)+\beta(u)H(-x))+\pa_x(f(u)H(x)+g(u)H(-x))=0$.
\end{remark}


\begin{figure}[htp]
\begin{center}
  \includegraphics[width=4in]{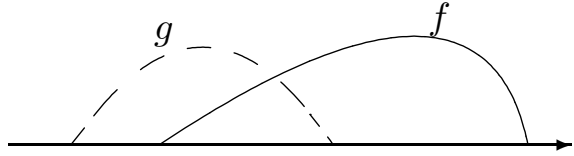}\\
  \caption{Functions $f$ (normal line) and $g$
(dashed line) satisfying the crossing condition.}
  \label{1}
  \end{center}
\end{figure}

First, we shall explain how to force the crossing condition and
existence of traces. We shall use the idea from \cite{Pan_08}. In
\cite{Pan_08}, the following problem was considered
\begin{equation}
\label{help1}
\begin{split}
\pa_t u + \pa_x f(\alpha(x,u))=&0,\\
u|_{t=0}=u_0(x),
\end{split}
\end{equation}where $\alpha$ is a function discontinuous in $x\in \R$ and strictly increasing with respect to $u$. Then, we can write:
$$
v=\alpha(x,u)  \Rightarrow u=\beta(x,v).
$$Problem \eqref{help1} becomes
\begin{equation}
\label{help2}
\begin{split}
\pa_t \beta(x,v) + \pa_x f(v)=&0,\\
v|_{t=0}=&\alpha(x,u_0).
\end{split}
\end{equation}Thus, the discontinuity in $x$ is removed out of the derivative in $x$, and we can apply standard vanishing viscosity approach:
\begin{equation}
\label{help3}
\begin{split}
\pa_t \beta(x,v_{\varepsilon}) + \pa_x f(v_{\varepsilon})=&\varepsilon \pa_{xx} v_{\varepsilon},\\
v|_{t=0}=& \alpha(x,u_0),
\end{split}
\end{equation} to obtain the sequence $(v_\varepsilon)$ strongly converging in $L^1_{loc}(\R\times \R^+)$ to a unique Kruzhkov admissible
weak solution $v$ of \eqref{help2} which immediately gives
uniqueness of appropriate weak solution to \eqref{help1}.

It is important to notice that the existence and uniqueness are
actually obtained thanks to the appropriate choice of the viscosity
term. Such choice enables the author to control the flux
corresponding to \eqref{help1}.

Using this observation, we shall propose new admissibility
conditions which will enable us to control the flux corresponding to
\eqref{oslo3} in an extent which will provide uniqueness in a rather
general situation. Informally speaking, we shall consider the
following vanishing viscosity regularization to \eqref{oslo3}:
\begin{equation}
\label{oslo3vv}
\begin{cases}
\pa_t u+\pa_x\left(H(x)f(u)+H(-x)g(u) \right)=\varepsilon \pa_{xx}(\tilde{\alpha}(u)H(x)+\tilde{\beta}(u)H(-x)),  \\
u|_{t=0}=u_0(x),
\end{cases}
\end{equation}where $\tilde{\alpha}: [a,b]\to [a',b']$ and
$\tilde{\beta}:[a,b]\to [a'',b'']$ are smooth strictly increasing
bijections.

Denote by $\alpha$ and $\beta$ the inverse functions of the
functions $\tilde{\alpha}$ and $\tilde{\beta}$, respectively.
Introducing the change of the unknown function:
$$
v=\tilde{\alpha}(u)H(x)+\tilde{\beta}(u)H(-x) \Rightarrow
u=\alpha(v)H(x)+\beta(v)H(-x),
$$and denoting $f_\alpha=f\circ \alpha$ and $g_\beta=g\circ\beta$, we have from \eqref{oslo3vv}:
\begin{equation}
\label{oslo3vv1}
\begin{cases}
\pa_t (\alpha(v)H(x)+\beta(v)H(-x))+\pa_x\left(H(x)f_\alpha(v)+H(-x)g_\beta(v) \right)=\varepsilon \pa_{xx} v, \\
v|_{t=0}=\tilde{\alpha}(u_0)H(x)+\tilde{\beta}(u_0)H(-x).
\end{cases}
\end{equation} So, instead of dealing with the flux $H(x)f(u)+H(-x)g(u)$, we deal with the new flux
$H(x)f_\alpha(v)+H(-x)g_\beta(v)$. As we shall see later, by
choosing appropriate functions $\alpha$ and $\beta$ we can always
make the new flux to satisfy "the crossing condition" at least in
the range of the solution (see Figure \ref{2} and Figure \ref{3} as
important special cases). Now, we can introduce the definition of
admissibility that we shall use.

\begin{definition}
\label{def-adm} Let $u$ be a weak solution to problem (\ref{oslo3}).
Let $\tilde{\alpha}:[a,b]\to [a', b']$ and $\tilde{\beta}:[a,b]\to
[a'', b'']$ be smooth strictly increasing bijections. Denote by
${\alpha}$ and ${\beta}$ the inverse functions to $\tilde{\alpha}$
and $\tilde{\beta}$, respectively.

We say that $u$ is an $(\alpha,\beta)$-entropy admissible solution
to \eqref{oslo3} if

(D.1.) $u\in L^\infty(\R^+\times \R)$ and $u(t,x)\in [a,b]$ for
almost every $(t,x)\in \R^+\times \R$;

(D.2) the function $v=\tilde{\alpha}(u)H(x)+\tilde{\beta}(u)H(-x)$
satisfies the following entropy condition for every fixed $\xi\in
{\bf R}$:
\begin{align}
 \label{sep827}
& \pa_t \Big\{{\rm sgn}(v-\xi)\Big[H(x)({\alpha}(v)-{\alpha}(\xi))
+H(-x)({\beta}(v)-{\beta}(\xi))\Big]\Big\}\\&+\pa_x
\Big\{{\rm sgn}(v-\xi)\Big[H(x)({f}_\alpha(v)-{f}_\alpha(\xi))+H(-x)({g}_\beta(v)-{g}_\beta(\xi))\Big]\Big\}\nonumber\\
&\qquad\qquad\qquad\qquad-|{f}_\alpha({\xi})-{g}_\beta({\xi})|\delta(x)
\leq 0,\nonumber
\end{align} where, as before, $f_\alpha=f\circ \alpha$ and
$g_\beta=g\circ \beta$.
\end{definition}

\begin{figure}[htp]
\begin{center}
  \includegraphics[width=5in]{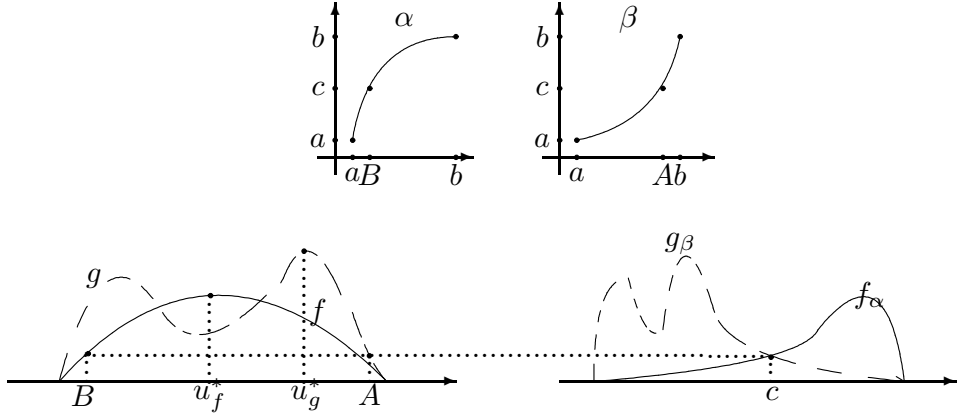}\\
  \caption{Functions $f$ (normal line) and $g$
(dashed line) on the left plot do not satisfy the crossing
condition. On the other hand, for appropriate (highly concave)
$\alpha$ and (highly convex) $\beta$, the functions $f_\alpha=f\circ
\alpha$ and $g_\beta=g\circ \beta$ on the right plot satisfy the
crossing conditions.}
  \label{2}
  \end{center}
\end{figure}


From the previous analysis, appealing on \cite{kar3}, we conclude
that we need only existence of traces to obtain the uniqueness. The
question of existence of traces is rather serious in itself
\cite{KV, Pan2, Vas}, but it was shown in \cite{Pan2} that they
exist practically in all relevant situations . In order to formulate
a necessary theorem, we need the notion of the quasi-solution.

\begin{definition}
We say that the function $u\in L^\infty(\R^d)$ is a quasi-solution
to the scalar conservation law
$$
{\rm div}_xF(u)=0, \ \ x\in \R^d,
$$ where $F=(F_1,\dots, F_d)\in C(\R^d;\R)$ if it satisfies for
every $\xi\in \R$:
\begin{align*}
{\rm div}_x {\rm sgn}(u-\xi)(F(u)-F(\xi))= \gamma_k \ \ {\rm in} \ \
{\cal D}'(\R^d),
\end{align*} where $\gamma_k$ is a locally bounded Borel measure.
\end{definition}

Next theorem can be found in \cite{Pan2}. We adapt it to our
situation.

\begin{theorem} \cite{Pan2}
\label{existenceoftraces} Let $h,f\in C(\R)$.

Suppose that the function $u$ is a quasi-solution to
\begin{equation*}
\pa_t h(u)+\pa_x f(u)=0,
\end{equation*} where the vector $(h,f)$ is such that the mappings $\lambda\mapsto h(\lambda)$ and $\lambda\mapsto f(\lambda)$ are
not constant on any non-degenerate interval.

Then, the function $u$ admits right and left strong traces at $x=0$.
\end{theorem}

Now, the situation with traces is clear and we need to cope with the
existence of a solution admissible in the sense of Definition
\ref{def-adm}.

In the case of a scalar conservation law with a smooth flux, the
proof of existence is based on the BV-estimates for a sequence of
solutions to the corresponding Cauchy problem regularized with the
vanishing viscosity. Such estimates are not available if the flux is
discontinuous. Therefore, we need to apply more subtle arguments
involving singular mapping \cite{temple}, local variation bounds
\cite{BGKT}, compensated compactness \cite{kar2, kar1, ken_chi,
Tar0}, difference schemes \cite{sid, kar3, ken_chi} or $H$-measures
\cite{HKM, Ger, Pan1, Tar}.

In general, using e.g. the compensated compactness, it is possible
to prove that the sequence $(u_\varepsilon)$ of solutions to
\eqref{oslo3vv} weakly converges to a weak solution $u$ of
\eqref{oslo3}. However, it is not possible to state that the weak
solution satisfies wanted admissibility conditions. In order to be
sure that $u$ is admissible, in principle, we need to prove that the
corresponding sequence $(u_\varepsilon)$ strongly converges strongly
in $L^1_{loc}(\R^+\times \R)$ to $u$ (still, not necessarily; see
\cite{Pan_15}) which, at least in the framework of the compensated
compactness (or the $H$-measures whose consequences we are going to
use), can be proved only by assuming the genuine nonlinearity
condition given by the following definition.

\begin{definition}
\label{non-deg}Let $h:\R^2\to \R$ and $f,g:\R\to \R$.

We say that the vector
$\left(h(x,\lambda),H(x)f(\lambda)+H(-x)g(\lambda)\right)$ is
genuinely nonlinear if for almost every $x\in \R$ and every
$(\xi_0,\xi_1)\in S^1$, $S^1\subset \R^2$ is two dimensional sphere,
the mapping
\begin{equation*}
(a,b)\ni \lambda \mapsto \xi_0
h(x,\lambda)+\xi_1\left(H(x)f(\lambda)+H(-x)g(\lambda)\right) ,
\end{equation*} is different from a
constant on any non-degenerate interval $(\alpha,\beta)\subset
(a,b)$.
\end{definition}

The latter condition provides the following theorem to hold.

\begin{theorem}
\cite{Pan1} \label{tcrucial} Assume that the vector
$(h(x,u),H(x)f(u)+H(-x)g(u))$,\\ $(x,u)\in \R\times \R$, is
genuinely nonlinear in the sense of Definition \ref{non-deg}.

Then, the following statement holds: \\
Each family $(v_{\varepsilon}(t,x))\in L^\infty(\R^+\times \R)$,
$a\leq v_\varepsilon \leq b$, $\varepsilon>0$, such that for every
$c\in \R$ the quantity
\begin{align}
\label{may257}
&\pa_t(H(v_{\varepsilon}-c)(h(x,v_{\varepsilon})-h(x,c)))\\&+\pa_x\left(
H(v_{\varepsilon}-c)(\left(H(x)(f(v_\varepsilon)-f(c))+H(-x)(g(v_\varepsilon)-g(c))\right))\right)
\nonumber
\end{align}
is precompact  in  $W_{\rm loc}^{-1,2}(\R^+\times \R)$, contains a
subsequence convergent in $L^1_{\rm loc}(\R^+\times \R)$.
\end{theorem}

So, our last obstacle is the genuine nonlinearity condition. In
order to overcome it we shall use an idea from \cite{kar1} which is
further developed in \cite{MA}. In \cite{kar1, MA}, existence of
solution to a Cauchy problem of type \eqref{oslo3} is proved.
Roughly speaking, the key point of the proof is based on a lemma
stating that if in \eqref{oslo3} we assume $u_0\in BV(\R)$, then,
for the sequence $(u_\varepsilon)$ of solutions to \eqref{oslo3vv},
it holds $\|\pa_t u_\varepsilon\|_{L^1(\R)}\leq const$ for every
fixed $t, \varepsilon \in \R^+$. This actually means that for any
function $h(x,\lambda)$, $x,\lambda\in \R$, which is Lipshitz
continuous in $\lambda$, it holds $\|\pa_t
h(x,u_\varepsilon)\|_{L^1(\R)}\leq const$ for every fixed $t,
\varepsilon \in \R^+$.

Next, it is not difficult to prove that it holds for the sequence
$(u_\varepsilon)$ of solutions to \eqref{oslo3vv}
\begin{equation*}
\pa_t(H(u_{\varepsilon}-c)(u_{\varepsilon}-c))+\pa_x\left(
H(u_{\varepsilon}-c)\left(H(x)(f(u_\varepsilon)-f(c))\!+\!H(-x)(g(u_\varepsilon)-g(c))\right)\right)
\end{equation*} is precompact  in $W_{\rm loc}^{-1,2}(\R^+\times \R^d)$. However, since $(|\pa_t u_\varepsilon|)$
is the sequence bounded in the space of Radon measures, we also
have:
\begin{equation*}
\begin{split}
&\pa_t(H(u_{\varepsilon}-c)\left(H(x)(h_R(u_\varepsilon)-h_R(c))+H(-x)(h_L(u_\varepsilon)-h_L(c))\right)\\&+\pa_x\left(
H(u_{\varepsilon}-c)\left(H(x)(f(u_\varepsilon)-f(c))+H(-x)(g(u_\varepsilon)-g(c))\right)\right)
\end{split}
\end{equation*} is precompact  in  $W_{\rm loc}^{-1,2}(\R^+\times \R^d)$ if $h_L,h_R\in {\rm Lip}(\R)$ (Lipschitz continuous functions).
Furthermore, if we choose $h_L$ and $h_R$ so that the vector
$(H(x)h_R(u)+H(-x)h_L(u), H(x)f(u)+H(-x)g(u))$ is genuinely
nonlinear, we can apply Theorem \ref{tcrucial} to conclude about
strong $L^1_{loc}$ precompactness of the family $(u_\varepsilon)$.
It is clear that a $L^1_{loc}$ limit along a subsequence of the
family $(u_\varepsilon)$ will represent wanted admissible weak
solution to \eqref{oslo3}. Furthermore, according to Theorem
\ref{existenceoftraces}, we infer about the existence of traces at
the interface $x=0$ for the previously constructed weak solution
which immediately gives uniqueness. Of course, it is not always
possible to choose $h_R$ and $h_L$ so that we have both, the genuine
nonlinearity and the crossing conditions fulfilled. Still, as we
shall see, using truncation functions $s_{l,k}(u)=\max\{l,\min\{k,u
\}\}$, $l<k$, $l,k\in \R$, (first used in \cite{Pan2} for this kind
of problems; see also \cite{HKM}), we are able to localize and thus
deal with the segments where the genuine nonlinearity is
unobtainable.

The paper is organized as follows.

In Section 1, we solve \eqref{oslo3} under additional assumptions on
the flux. We find the section important since it sheds (another)
light on paper \cite{kar4} where the crossing condition is bypassed
 by using so called adapted entropies (see \cite{AP}). We show
that admissibility conditions that we introduced in Definition
\ref{def-adm} can be considered as a generalization of the approach
from \cite{kar4}, which is actually an explanation how adapted
entropies enabled avoiding (or maybe better to say forced) the
crossing conditions.

In Section 2, by passing to the measure valued solution concept
\cite{Dpe}, we show existence and uniqueness in the general
situation.

\section{New entropy admissibility conditions}

The basic purpose of the section is to explain connection between
our $(\alpha,\beta)$-entropy solutions and the entropy solutions of
type $(A,B)$ used in \cite{kar4}. Furthermore, we find that this
section represents a good introduction into the general situation
considered in Section 3.

We shall consider here \eqref{oslo3} under the additional
assumptions that the mappings
\begin{equation}
\label{bergen_11} \lambda\mapsto f(\lambda), \ \ \lambda\mapsto
g(\lambda)
\end{equation} are nonconstant and strictly positive on any subinterval of the interval
$(a,b)$ (notice that this assumption is weaker than the appropriate
assumption \cite[(1.2)]{kar4} which demands a genuine nonlinearity
of $f$ and $g$).

To proceed, let us briefly recall the concept from \cite{kar4}.
First, we need the function $c^{AB}$ (see \cite[(11)]{kar4}):
$$
c^{AB}(x)=\begin{cases} A, & x\leq 0\\
B, & >0
\end{cases}.
$$
In \cite{kar4}, the function $c^{AB}$ is used to form the function
$u\mapsto |u- c^{AB}(x)|$ which is an example of what is in
\cite{AP} called an adapted entropy. Still, in \cite{AP}, the
existence of infinitely many adapted entropies was necessary to
prove uniqueness (see also \cite{Pan_08}) while in \cite{kar4} only
the entropy $u\mapsto |u- c^{AB}(x)|$ was sufficient (together with
the classical Kruzhkov entropies out of the interface). The function
$c^{AB}$ is called a connection if it represents a weak solution to
\eqref{oslo3}, i.e. if $f(B)=g(A)$ (see Remark \ref{refrem} for a
more precise explanation). We remark that the notion of the
connection originated from \cite{sid_2}. The following admissibility
conditions were used in \cite{kar4}:

\begin{definition}\cite[Definition
3.1.]{kar4} \label{kenneth_2}  (Entropy solution of type $(A,B)$).
A measurable function $u:\R^+\times \R\to \R$, representing a weak
solution to \eqref{oslo3} is an entropy solution of type $(A,B)$ if
it satisfies the following conditions:

(D.1) $u \in L^\infty(\R^+\times\R)$; $u(t,x) \in [a,b]$ for a.e.
$(t,x) \in \R^+\times \R$.

(D.2) For any test function $0\leq \varphi\in {\cal D}([0, T )\times
\R)$, $T>0$, which vanishes for $x\geq 0$, and any $\xi\in \R$, the
following holds:
\begin{equation*}
\int_0^T\int_{\R}\left( |u- \xi|\varphi_t + {\rm sgn}(u - \xi)( f(u)
-f(\xi))\varphi_x \right)dx dt +\! \int_{\R} |u_0 - \xi|
\varphi(0,x)dx\geq  0,
\end{equation*} and for any test function $0\leq \varphi\in {\cal D}([0, T )\times \R)$, $T>0$, which vanishes
for $x\leq 0$
\begin{equation*}
\int_0^T\int_{\R}\left( |u- \xi|\varphi_t + {\rm sgn}(u - \xi)( g(u)
-g(\xi))\varphi_x \right)dx dt + \int_{\R} |u_0 - \xi|
\varphi(0,x)dx\geq 0,
\end{equation*}

(D.3) The following Kruzhkov-type entropy inequality holds for any
test function $0 \leq \varphi \in {\cal D}([0, T )\times \R)$,
$T>0$,
\begin{equation*}
\begin{split}
 \int_0^T&\int_{\R}\Big( |u- c^{AB}(x)|\varphi_t \\
&+ \!{\rm sgn}(u\! -\! c^{AB}(x))(H(x)(f(u)\! -\!f(A))\!+\!
H(-x)(g(u)\! -\!g(B)))\varphi_x \Big)dx dt \\& + \int_{\R} |u_0 -
c^{AB}(x)| \varphi(0,x)dx\geq  0.
\end{split}
\end{equation*}
\end{definition}
In the next theorem, we state that the $(\alpha,\beta)$-entropy
admissible solution from Definition \ref{def-adm} is, under certain
conditions, at the same time an entropy solution of type $(A,B)$
from Definition \ref{kenneth_2}. In Remark \ref{refrem} after the
theorem, we shall explain why such conditions are always fulfilled
in the case of the flux given in \cite{kar4}.

\begin{theorem}
\label{equiv} Assume that the function $u$ is an
$(\alpha,\beta)$-entropy admissible solution to \eqref{oslo3} in the
sense of Definition \ref{def-adm} where ${\alpha}$ and ${\beta}$
satisfy:
\begin{itemize}

\item $\alpha,\beta:[a,b] \to [a,b]$;

\item there exists $c\in (a,b)$ such that ${\alpha}(c)=B$ and
${\beta}(c)=A$ where $f(B)=g(A)$;


\item the functions $f\circ \alpha$ and $g\circ \beta$ satisfy the
crossing conditions.

\end{itemize} Then, the $(\alpha,\beta)$-entropy admissible solution
to \eqref{oslo3} is at the same time the entropy solution of type
$(A,B)$.
\end{theorem}

\begin{proof}
First, notice that, according to the choice of $\alpha$ and $\beta$,
the function $c^{AB}$ will represent an $(\alpha,\beta)$-entropy
admissible solution to \eqref{oslo3} in the sense of Definition
\ref{def-adm}. Taking another $(\alpha,\beta)$-entropy admissible
solution to \eqref{oslo3}, say $u={\alpha}(v)H(x)+{\beta}(v)H(-x)$,
and applying the procedure from \cite{kar3} leading to
\cite[(2.34)]{kar3} (keep in mind that $f_\alpha$ and $g_\beta$
satisfy the crossing conditions), we reach to the following (well
known) relation:
\begin{align}
\label{comp_1} &\pa_t {\rm
sgn}(v-c)\left(H(x)({\alpha}(v)-\alpha(c))+H(-x)({\beta}(v)-\beta(c))
\right)\\& + \pa_x {\rm
sgn}(v-c)\left(H(x)(f_{\alpha}(v)-f_\alpha(c))+H(-x)(g_{\beta}(v)-g_\beta(c))
\right)\leq 0. \nonumber
\end{align} Since $\alpha$ and $\beta$ as well as their inverses $\tilde{\alpha}$ and $\tilde{\beta}$ are increasing bijections, it holds
$$
{\rm sgn}(v-c)={\rm
sgn}\left((\tilde{\alpha}(u)-\tilde{\alpha}(B))H(x)+(\tilde{\beta}(u)-\tilde{\beta}(A))H(x)\right)={\rm
sgn}(u-c^{AB}).
$$From here, we see that \eqref{comp_1} is actually condition (D.3)
from Definition \ref{kenneth_2} meaning that the
$(\alpha,\beta)$-entropy admissible solution $u$ is, at the same
time, an entropy solution of type $(A,B)$ (conditions (D.1.) and
(D.2.) from Definition \ref{kenneth_2} are easily checked).
\end{proof}

\begin{remark}
\label{refrem} The notion of connection used in \cite{kar4} relied
on the case when the functions $f$ and $g$ forming the flux in
\eqref{oslo3} were such that they admit unique local maxima points
$u^*_f\in (a,b)$ and $u^*_g\in (a,b)$, respectively. Then, the pair
$(A,B)$ is called a connection if
\begin{equation}
\label{conn} f(B)=g(A) \ \ \text{with} \ \ u_g^*\leq A \leq b, \ \
a\leq B \leq u_f^*.
\end{equation} In this
case, we can always find functions $\alpha$ and $\beta$ such that
conditions of Theorem \ref{equiv} are satisfied.

Indeed, assume that $u^*_f<u^*_g$ (other two situations
$u^*_f>u^*_g$ and $u^*_f=u^*_g$ can be resolved similarly). Denote
by $\tilde{\alpha}$ and $\tilde{\beta}$ inverse functions to the
functions $\alpha$ and $\beta$, respectively. Choose
$\tilde{\alpha}$ and $\tilde{\beta}$ on the intervals $[u^*_f,b]$
and $[a,u_g^*]$ to be linear and such that $\tilde{\alpha}(u^*_f)
>c> \tilde{\beta}(u^*_g)$ (see Figure \ref{2}; the situation plotted there is more
general but completely analogical with the one we are considering at
the moment).

To extend the function $\tilde{\alpha}$ in the interval $[a,u_f^*]$,
we will construct its inverse $\alpha$ in the interval $[a,c]$. Take
an arbitrary decreasing function $\hat{\alpha}$ connecting the
points $(a,0)$ and $(c,f_\alpha(c))$ such that $\hat{\alpha}\leq
g_\beta$ on $[a,c]$. This is always possible since $g_\beta >0$ on
$(a,c)$; for instance, we can take $\hat{\alpha}$ to be the convex
hull of $g_\beta$ on $[a,c]$. Then, put
$\alpha=f^{-1}\circ\hat{\alpha}$ i.e. $\tilde{\alpha}=\alpha^{-1}$
on $[a,u_f^*]$ (this is permitted since $f^{-1}$ is monotonic on
$(a,c)$). We choose $\tilde{\beta}$ on $[u_g^*,b]$ in the completely
same manner (see Figure \ref{2} for further clarification). It is
clear that $\alpha=\tilde{\alpha}^{-1}$ and
$\beta=\tilde{\beta}^{-1}$ chosen in such a way satisfy conditions
of Theorem \ref{equiv}.

Actually, from the latter discussion, we can conclude that the
conditions given in Theorem \ref{equiv} are a generalization of the
notion of connection. More precisely, we can say that a pair $(A,B)$
is a connection if there exist functions $\alpha$ and $\beta$
satisfying conditions of Theorem \ref{equiv}. As we shall see in
Theorem \ref{th-main}, such conditions provide existence and
uniqueness of the $(\alpha,\beta)$-entropy admissible solution to
\eqref{oslo3}. In particular, the function $B H(x)+A H(-x)$ will be
the $(\alpha,\beta)$-entropy admissible shock.

Also, remark that conditions \eqref{conn} can be naturally
generalized by assuming that
\begin{equation}
\label{conn1}
 f(B)=g(A) \ \ \text{with} \ \ A\in (u^*_g, b), \;  B\in (a, u^*_f),
\end{equation} where $u^*_g$ and $u^*_f$ are the rear right local maximum of
the function $g$ and the rear left local maximum of the function
$f$, respectively. Repeating the procedure from the beginning of the
remark, we can find the function $\alpha$ and $\beta$ such that
conditions of Theorem \ref{equiv} are satisfied (see Figure
\ref{2}).

Finally, notice that if $A=u^*_g$ and $B=u^*_f$, we cannot state
that the functions $\alpha$ and $\beta$ satisfying conditions of
Theorem \ref{equiv} exist (for instance, if the functions $f$ and
$g$ have several local maxima, and all of them have the same
values).

\end{remark}

The following theorem is the main theorem of the section:

\begin{theorem}
\label{th-main} There exists a pair of function $(\alpha,\beta)$
from Definition \ref{def-adm} such that there exists a unique
$(\alpha,\beta)$-entropy admissible solution to \eqref{oslo3}.

For such $\alpha$ and $\beta$ any two $(\alpha,\beta)$-entropy
admissible solutions $u$ and $v$ to \eqref{oslo3} satisfy
\eqref{new_version}.
\end{theorem}

Before we prove the theorem, we shall need several auxiliary
statements and explanations.

In order to construct an $(\alpha,\beta)$-entropy admissible
solution to \eqref{oslo3}, we use a non-standard vanishing viscosity
approximation with regularized flux. First, introduce the following
change of the unknown function $u$:
$$
u(t,x)=\tilde{\alpha}(v(t,x))H(x)+\tilde{\beta}(v(t,x))H(-x),
$$ for increasing functions $\tilde{\alpha}, \tilde{\beta}:[a,b]\to [a,b]$. Denote by $\alpha=\tilde{\alpha}^{-1}$
and $\beta=\tilde{\beta}^{-1}$. Equation \eqref{oslo3} becomes:
\begin{equation}
\label{oslo3'} \pa_t \left( H(x)\alpha(v)+H(-x)\beta(v)
\right)+\pa_x\left(H(x)f_\alpha(v)+H(-x)g_\beta(v) \right)=0.
\end{equation}
Then, take the following regularization of the Heaviside function
$H$, $H_{\varepsilon}(x)=\int_{-\infty}^{x/\varepsilon}\omega(z)dz$,
where $\omega$ is a smooth even compactly supported function with
total mass one. Let $\chi_\varepsilon$ be a smooth function equal to
one in the interval $(-1/\varepsilon,1/\varepsilon)$ and zero out of
the interval $(-2/\varepsilon,2/\varepsilon)$. Consider the
following regularized problem:
\begin{equation}
\label{aug2817} \begin{split} &\pa_t
\left(H_{\varepsilon}(x)\alpha(v_\varepsilon)+H_{\varepsilon}(-x)\beta(v_\varepsilon)
\right)\\&+\pa_x\left(H_{\varepsilon}(x)f_\alpha(v_\varepsilon)+H_{\varepsilon}(-x)g_\beta(v_\varepsilon)
\right)=\varepsilon\pa_{xx}v_\varepsilon\\
&v_\varepsilon\big|_{t=0}=(\tilde{\alpha}(u_0)H(x)+\tilde{\beta}(u_0)H(-x))\star
\frac{1}{\varepsilon}\omega(\cdot/\varepsilon)\chi_\varepsilon(x)
\end{split}
\end{equation}

Obviously, for every fixed $\varepsilon>0$ quasilinear parabolic
Cauchy problem \eqref{aug2817} will have a unique smooth solution
$v_{\varepsilon}$.

Since $\alpha$ and $\beta$ are strictly increasing functions which
map interval $[a,b]$ into itself, slightly modifying the methodology
from \cite{kar1}, we obtain the following three lemmas.

\begin{lemma}\cite[Lemma 4.1]{kar1}
\label{Lbbound}[$L^\infty$-bound] There exists constant $c_0>0$ such
that for all $t\in(0,T)$,
$$
\|v_\varepsilon(t,\cdot)\|_{L^\infty(\R)}\leq c_0.
$$More precisely,
$$
a\leq v_\varepsilon \leq b.
$$
\end{lemma}

\begin{lemma} \cite[Lemma 4.2]{kar1}
\label{regintime} [Lipshitz regularity in time] Assume that the
initial function $u_0$ from \eqref{oslo3} has bounded variation.
Then, there exists constant $c_1$, independent of $\varepsilon$,
such that for all $t>0$,
$$
\int_{\R}|\partial_t v_\varepsilon(\cdot, t)|\,dx\leq c_1.
$$
\end{lemma}

\begin{lemma} \cite[Lemma 4.3]{kar1}
\label{entrbound}[Entropy dissipation bound] There exists a constant
$c_2$ independent from $\varepsilon$ such that
$$
\varepsilon\int_{\R}\left( \pa_xv_\varepsilon(t,x)\right)^2dx\leq
c_2,
$$
for all $t>0$.
\end{lemma}

To proceed, we need Murat's lemma:
\begin{lemma}
\cite{Eva} \label{mur_lemma} Assume that the family
$(Q_\varepsilon)$ is bounded in $L^p(\Omega)$, $\Omega \subset
\R^d$, $p>2$.

Then,
$$
(\Div Q_\varepsilon)_\varepsilon\in W^{-1,2}_{\rm c,loc} \ \
\text{if} \ \ \Div Q_\varepsilon=p_\varepsilon+q_\varepsilon,
$$
with $(q_\varepsilon)_\varepsilon\in W^{-1,2}_{\rm c,loc}(\Omega)$
and $(p_\varepsilon)_\varepsilon\in {\cal M}_{b,loc}(\Omega)$.
\end{lemma}

Now, we can prove a crucial lemma for obtaining the existence of the
$(\alpha,\beta)$-entropy admissible solution to \eqref{oslo3}.

\begin{lemma}
\label{11} Denote for a fixed $\xi\in \R$:
\begin{equation}
\label{entr_adapt}
\begin{split}
q(x,\la)&=H(\la\!-\!\xi)\Big(H(x)(f_\alpha(\lambda)\!-\!f_\alpha(\xi))\!+\!H(-x)(g_\beta(\lambda)\!-\!g_\beta(\xi))\Big),\\
\bar{q}(x,\lambda)&=H(\la\!-\!\xi)\Big(H(x)(f_\alpha^2(\lambda)\!-\!f_\alpha^2(\xi))\!+\!H(-x)(g_\beta^2(\lambda)\!-\!g_\beta^2(\xi))\Big),\\
q_{\alpha,\beta}(x,\lambda)&=H(\la\!-\!\xi)\Big(H(x)(\alpha(\lambda)\!-\!\alpha(\xi))\!+\!H(-x)(\beta(\lambda)\!-\!\beta(\xi))\Big).
\end{split}
\end{equation}
If the initial function $u_0$ from \eqref{oslo3} has bounded
variation then the family
\begin{equation}
\label{mur}
\begin{split}
&\pa_t\bar{q}(x,v_\varepsilon)+\pa_x{q}(x,v_\varepsilon), \ \
\varepsilon>0,
\end{split}
\end{equation} is precompact  in  $W_{\rm loc}^{-1,2}(\R^+\times \R)$.
\end{lemma}
{\bf Proof:}

Denote  $\eta'(\la)=H(\la-\xi)$. Define the entropy flux which
corresponds to \eqref{aug2817}:
\begin{equation*}
\begin{split}
q^{\varepsilon}(x,\la)\!&=\!H(\la\!-\!\xi)
\Big(H_{\varepsilon}(x)(f_\alpha(\lambda)\!-\!f_\alpha(\xi))\!+\!H_{\varepsilon}(-x)(g_\beta(\lambda)\!-\!g_\beta(\xi))\Big),\\
q_{\alpha,\beta}^{\varepsilon}(x,\la)\!&=\!H(\la\!-\!\xi)
\Big(H_{\varepsilon}(x)(\alpha(\lambda)\!-\!\alpha(\xi))\!+\!H_{\varepsilon}(-x)(\beta(\lambda)\!-\!\beta(\xi))\Big).
\end{split}
\end{equation*}

Denote $\delta_{\varepsilon}(x)=H'_{\varepsilon}(x)$, $i=1,2$. After
multiplying \eqref{aug2817} by $\eta'(v_\varepsilon)$, we obtain in
the sense of distributions:
\begin{align}
\label{mur_1} &\pa_t
q_{\alpha,\beta}^\varepsilon(x,v_\varepsilon)+\pa_x
q^{\varepsilon}(x,v_\varepsilon)\\&=
\left(\delta_{\varepsilon}(x)f_\alpha(\xi)-\delta_{\varepsilon}(x)g_\beta(\xi)\right)
+\varepsilon( \pa_x(v_{\varepsilon x} \eta'(v_\varepsilon))-(v_{\varepsilon x})^2\eta''(v_\varepsilon)\nonumber\\
&\leq\delta_{\varepsilon}(x) \left(f_\alpha(\xi)-g_\beta(\xi)\right)
+\varepsilon( \pa_x(v_{\varepsilon x} \eta'(v_\varepsilon)).
\nonumber
\end{align}

From here, according to the Schwartz lemma for non-negative
distributions, we conclude that there exists a positive Radon
measure $\mu^\varepsilon_\xi(t,x)$ such that:
\begin{align}
\label{mur_2} &\pa_t q_{\alpha,\beta}^{\varepsilon}(x,v_\varepsilon)
+\pa_xq^{\varepsilon}(x,v_\varepsilon)\\&=
\delta_{\varepsilon}\left(f_\alpha(\xi)-g_\alpha(\xi)\right)
+\varepsilon( \pa_x(v_{\varepsilon
x}\eta'(v_\varepsilon))-\mu^\varepsilon_\xi(t,x). \nonumber
\end{align}

Rewrite expression \eqref{mur_2} in the form:
\begin{align}
\label{mur_3}
&\pa_t\bar{q}(x,v_\varepsilon)+\pa_x q(x,v_\varepsilon)\\
&=\pa_t\left(\bar{q}(x,v_\varepsilon)-q^\varepsilon_{\alpha,\beta}(x,v_\varepsilon)\right)+\pa_x\left(
q^{\varepsilon}(x,v_\varepsilon)-q(x,v_\varepsilon)
\right)\nonumber\\&+
\delta_{\varepsilon}\left(f_\alpha(\xi)-g_\beta(\xi)\right)
+\varepsilon( \pa_x(v_{\varepsilon
x}\eta'(v_\varepsilon))-\mu^\varepsilon_\xi(t,x). \nonumber
\end{align} Since, clearly, $q^{\varepsilon}(x,v_\varepsilon)-q(x,v_\varepsilon)\to 0$
 as $\varepsilon\to 0$ pointwisely, we derive the statement of the lemma from the Lebesgue
dominated convergence theorem, Lemmas \ref{Lbbound}-\ref{entrbound},
and Lemma \ref{mur_lemma}. For details please consult \cite[Theorem
2.6.]{MA}  $\Box$

From Lemma \ref{11} and Theorem \ref{tcrucial}, it is easy to prove
that for any choice of the functions $\alpha$ and $\beta$ from
Definition \ref{def-adm} there exists an $(\alpha,\beta)$-entropy
admissible solution to \eqref{oslo3} provided $u_0\in BV(\R)$:

\begin{theorem}
\label{19} Assume that $u_0\in BV(\R)$, and that $f$ and $g$ satisfy
\eqref{bergen_11}, where $u_0$, $f$ and $g$ are given in
\eqref{oslo3}. For any bijections $\alpha,\beta:[a,b]\to [a,b]$ from
Definition \ref{def-adm} there exists an $(\alpha,\beta)$-entropy
admissible weak solution to \eqref{oslo3}.
\end{theorem}
{\bf Proof:} First, notice that the vector
$(\bar{q}(x,\lambda),{q}(x,\lambda))$ from \eqref{entr_adapt} is
genuinely nonlinear. Indeed, for $x>0$ the vector reduces to
$(f_\alpha^2(\lambda), f_\alpha(\lambda))$ and this is obviously
genuinely nonlinear vector according to \eqref{bergen_11}.
Similarly, we conclude about the genuine nonlinearity for $x<0$.

Now, from Theorem \ref{tcrucial} and Lemma \ref{11}, we conclude
that the family $(v_\varepsilon)$ of solutions to \eqref{aug2817} is
strongly precompact in $L^1_{loc}(\R^+\times \R)$. Denote by $v$ the
$L^1_{loc}(\R^+\times \R)$ limit along a subsequence of the family
$(v_\varepsilon)$. Clearly, $u=\alpha(v)H(x)+\beta(v)H(-x)$ will
represent the $(\alpha,\beta)$-entropy admissible solution to
\eqref{oslo3}. $\Box$

Now, we can prove the main theorem of the section.

{\bf Proof of Theorem \ref{th-main}:} We need to find the functions
$\alpha$ and $\beta$ so that the functions $f_\alpha$ and $g_\beta$
satisfy the crossing conditions. As explained in Remark
\eqref{refrem}, we choose the points $A, B\in (a,b)$ satisfying
\eqref{conn1}, and construct the functions $\alpha$ and $\beta$ so
that for appropriate $c\in (a,b)$ it holds $\alpha(c)=B$,
$\beta(c)=A$, and $f_\alpha \geq g_\beta$ on $[c,b]$, and $f_\alpha
\leq g_\beta$ on $[a,c]$ which is nothing else but the crossing
condition for $f_\alpha$ and $g_\beta$.

Next, assume that $u_0\in BV(\R)$ and denote by $u$ the
$(\alpha,\beta)$-entropy admissible solution to \eqref{oslo3} (it is
given by Theorem \ref{19}). Notice that from the construction (it is
enough to let $\varepsilon\to 0$ in \eqref{mur_3}) and Lemmas
\ref{Lbbound}-\ref{entrbound}, it follows that the function
$v=\tilde{\alpha}(u)H(x)+\tilde{\beta}(u)H(-x)$ is, at the same
time, a quasi-solution to the equation:
\begin{equation}
\label{bgr_10}
\begin{split}
&\pa_t\Big(H(x)f_\alpha^2(v)\!+\!H(-x)g_\beta^2(v)\Big)\!+\!\pa_x\Big(H(x)f_\alpha(v)\!+\!H(-x)g_\beta(v)\Big)=0,\\
\end{split}
\end{equation} Since the vector
$(H(x)f_\alpha^2(\lambda)+H(-x)g_\beta^2(\lambda),H(x)f_\alpha(\lambda)+H(-x)g_\beta(\lambda))$
is genuinely nonlinear (see \eqref{bergen_11}), according to Theorem
\ref{existenceoftraces}, the function $v$ admits strong traces at
the interface $x=0$.

Similarly, from the construction again and according to the choice
of the function $\alpha$ and $\beta$, we see that $v$ is an entropy
admissible solution in the sense of Definition \ref{kar1} to the
Cauchy problem
\begin{equation}
\label{aux_mar}
\begin{split}
&\pa_t \left( {\alpha}(v)H(x) + {\beta}(v)H(-x)\right)+ \pa_x\left(H(x){f}_\alpha(v)+ H(-x){g}_\alpha(v) \right)=0,\\
&v|_{t=0}=\alpha(u_0)H(x)+\beta(u_0)H(-x),
\end{split}
\end{equation} where $f_\alpha$ and $g_\beta$ satisfy the
crossing condition.

According to Theorem \ref{kenn}, we conclude that $v$ is a unique
entropy admissible solution to \eqref{aux_mar} in the sense of
Definition \ref{kar1} implying that
$u={\alpha}(v)H(x)+{\beta}(v)H(-x)$ is a unique
$(\alpha,\beta)$-entropy admissible solution to \eqref{oslo3}.

Now, assume that $u_0\notin BV(\R)$. Approximate the function
$u_{0}$ by a sequence $(u_{0\delta})\in BV(\R)$ so that
$$
u_0-u_{0\delta}\to 0 \ \ {\rm as} \ \ \delta\to 0
$$ strongly in $L^1_{loc}(\R)$. Then, we find a unique
$(\alpha,\beta)$-entropy admissible solution $u_\delta$ to
\eqref{oslo3} where $u|_{t=0}=u_{0\delta}$ (given $\alpha$ and
$\beta$ for which we have uniqueness i.e. such that $f_\alpha$ and
$g_\beta$ satisfy the crossing conditions). According to Theorem
\ref{th-main}, the family $(u_\delta)$ satisfy the following
stability relation:
$$
\int_0^T\int_{-R}^R|u_{\delta_1}-u_{\delta_2}|dxdt\leq
CT\int_{-\bar{R}}^{\bar{R}}|u_{0\delta_1}-u_{0\delta_2}|dx,
$$ where $R$ and $T$ are arbitrary positive constants, and
$C, \bar{R}$ are constants depending on $R$, the functions $f$, $g$,
$\alpha$ and $\beta$. Since the right-hand side of the latter
expression is uniformly small with respect to $\delta_1$ and
$\delta_2$, from the Cauchy criterion we conclude that there exists
$u\in L^1_{loc}$ such that $u_\delta\to u$ strongly in
$L^1_{loc}(\R^d)$. Clearly, the function $u$ will represent an
$(\alpha, \beta)$-entropy admissible solution to \eqref{oslo3}.

Since, according to \eqref{bergen_11} and Theorem
\ref{existenceoftraces}, the function $u$ admits strong traces at
$x=0$, we conclude that it must be a unique $(\alpha, \beta)$-
entropy admissible solution to \eqref{oslo3}. $\Box$

\section{General case}

At the beginning, notice that there are many examples of fluxes from
\eqref{oslo3} when we can not apply the procedure from the previous
section (see Figure 3). Therefore, in this section, we shall
demonstrate how to apply the $(\alpha,\beta)$-entropy admissibility
concept on \eqref{oslo3} in a general case. More precisely, we shall
only assume that $f,g\in C^1({\bf R})$ are such that
$f(a)=f(b)=g(a)=g(b)=0$, and, for simplicity, that there exists a
finite number of intervals $(a_{r_j},a_{r_{j}+1})$, $j=1,\dots,k_r$,
and $(b_{l_i},b_{l_{i}}+1)$, $i=1,\dots,k_l$, $k_l, k_r\in \N$, such
that the mappings

\begin{equation}
\label{cond_brg}
\begin{split}
&\text{ $\lambda\mapsto g(\lambda)$ and $\lambda
\mapsto f(\lambda)$ are constant on the intervals}\\
&\text{ $(b_{l_i},b_{l_{i}+1})$, $i=1,\dots,k_l$, and
$(a_{r_j},a_{r_{j}+1})$, $j=1,\dots,k_r$, respectively.}
\end{split}
\end{equation} For a convenience, assume that $[a,b]=\ccup\limits_{i=1}^{n_r}
[a_i,a_{i+1})$ and $[a,b]=\ccup\limits_{i=1}^{n_l} [b_i,b_{i+1})$,
where $n_l,n_r\in \N$, and $a_1=b_1=a$, and $a_{n_r}=b_{n_l}=b$.

\begin{figure}[htp]
\begin{center}
  \includegraphics[width=2in]{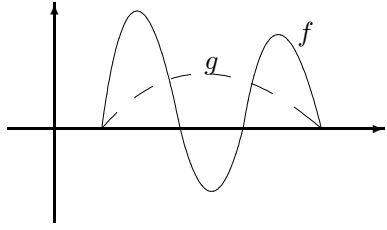}\\
  \caption{Functions $f$ (normal line) and $g$
(dashed line) do not satisfy the crossing condition and there exist
no increasing bijections $\alpha, \beta:[a,b]\to [a,b]$ such that
$f\circ\alpha$ and $g\circ\beta$ satisfy the crossing conditions.}
  \label{referee}
  \end{center}
\end{figure}


We shall need the notion of Young measures (we will be highly
selective and, for an application of Young measures in conservation
laws, address a reader on famous paper \cite{Dpe}).

\begin{theorem}\cite{Pablo}
\label{pablo} Assume that the sequence $(u_{\varepsilon_k})$ is
uniformly bounded in \\ $L^{p}_{loc}(\R^+\times \R^{d}))$, $p\geq
1$. Then, there exists a subsequence (not relabeled)
$(u_{\varepsilon_k})$ and a family of probability measures
$$
\nu_{t,x}\in {\cal M}(\R), \ \ (t,x)\in\R^+\times\R^d
$$
such that the limit
$$
\bar{g}(t,x):=\lim_{k\to \infty}g(u_{\varepsilon_k}(t,x))
$$
exists in the distributional sense for all $g\in C(\R)$. The limit
is represented by the expectation value
$$
\bar{g}(t,x)=\int_{\R^+\times\R^d}g(\lambda)d\nu_{t,x}(\lambda),
$$
for almost all points $(t,x)\in\R^+\times\R^d$.

We refer to such a family of measures $\nu=(\nu_{(t,x)})_{(t,x)\in
\R^+\times \R}$ as the Young measure associated to the sequence
$(u_{\varepsilon_k})_{k\in {\bf N}}$.

Furthermore,
$$
u_{\varepsilon_k}\to u \ \ {\rm in} \ \ L^r_{\rm
loc}(\R^{+}\times\R^{d}), \; 1\leq r<p
$$
if and only if
$$
\nu_{t,x}(\lambda)=\delta(\lambda-{u(t,x)}) \ \ a.e. \ \ (t,x)\in
\R^+\times\R,
$$ where $\delta$ is the Dirac distribution.
\end{theorem}

Introduce the truncation operator $s_{l,k}(u)=\max\{l,\min\{k,u
\}\}$, $l<k$, $l,k\in \R$. The following important lemma holds.

\begin{lemma}
\label{l20} Denote by $(v_\varepsilon)$ family of solutions to
\eqref{aug2817} where $u_0\in BV(\R;[a,b])$ and $f$ and $g$ satisfy
\eqref{cond_brg}. Assume that the mapping $\lambda\mapsto
f(\lambda)$ is not constant on any subinterval of an interval
$(l,k)$. Then, the sequence $(H(x)s_{l,k}(v_\varepsilon))$ is
strongly precompact in $L^1_{loc}(\R^+\times\R)$.

Similarly, if the mapping $\lambda\mapsto g(\lambda)$ is not
constant on any subinterval of an interval $(l,k)$. Then, the
sequence $(H(-x)s_{l,k}(v_\varepsilon))$ is strongly precompact in
$L^1_{loc}(\R^+\times\R)$.
\end{lemma}

\begin{proof}

Notice that from Lemma \ref{11}, it follows that for the family of
functions $v_\varepsilon$ and any $k,l\in \R$, the families
\begin{equation}
\label{budva_1}
\begin{split}
&\pa_t\bar{q}(x,H(x)s_{l,k}(v_\varepsilon))+\pa_x
q(x,H(x)s_{l,k}(v_\varepsilon)) \ \ {\rm and} \\
&\pa_t\bar{q}(x,H(-x)s_{l,k}(v_\varepsilon))+\pa_x
q(x,H(-x)s_{l,k}(v_\varepsilon)),
\end{split}
\end{equation}where the functions
$\bar{q}, q$ given by \ref{entr_adapt}, are strongly precompact in
$W^{-1,2}_{loc}(\R^+\times\R)$. Indeed, notice that
\begin{equation}
\label{budva_2}
\begin{split}
q(x,H(x)s_{l,k}(v_\varepsilon))&=H(x)q(x,s_{l,k}(v_\varepsilon))-H(-\xi)H(-x)(g_\beta(0)-g_\beta(\xi))\\
\bar{q}(x,H(x)s_{l,k}(v_\varepsilon))&=H(x)\bar{q}(x,s_{l,k}(v_\varepsilon))-H(-\xi)H(-x)(g^2_\beta(0)-g^2_\beta(\xi))
\end{split}.
\end{equation} Since $\pa_t\bar{q}(x,s_{l,k}(v_\varepsilon))+\pa_x q(x,s_{l,k}(v_\varepsilon))$ is strongly precompact
in $W^{-1,2}_{loc}(\R^+\times\R)$ if
$\pa_t\bar{q}(x,v_\varepsilon)+\pa_x q(x,v_\varepsilon)$ is (see
\cite[Theorem 6]{Pan1}), we conclude from \eqref{budva_2} that
\eqref{budva_1} holds.

Furthermore, notice that if the mapping $ \lambda\mapsto f(\lambda)$
 is not constant on any subinterval
of an interval $(k,l)$ then the vector
$(\bar{q}(x,\lambda),{q}(x,\lambda))$ from \eqref{entr_adapt} is
genuinely nonlinear on the interval $(l,k)$ and $x>0$. Indeed, for
$x>0$ the vector reduces to $(f_\alpha^2(\lambda),
f_\alpha(\lambda))$ and this is obviously genuinely nonlinear vector
since, due to the assumptions of the lemma, for any $\xi_0,\xi_1\in
\R$, it holds $\xi_0f^2(\lambda)\neq \xi_1 f(\lambda)$ for a.e.
$\lambda\in (k,l)$. Now, from Theorem \ref{tcrucial} and Lemma
\ref{11}, we conclude that the family $(H(x)s_{k,l}(v_\varepsilon))$
is strongly precompact in $L^1_{loc}(\R^+\times \R)$.

In the completely same way, we conclude that the family
$(H(-x)s_{k,l}(v_\varepsilon))$ is strongly precompact in
$L^1_{loc}(\R^+\times \R)$ if the mapping $\lambda\mapsto
g(\lambda)$ is different from a constant on every subinterval of the
interval $(k,l)$. \end{proof}

Next lemma deals with precompactness properties of the family
$(f(v_\varepsilon)H(x)+g(v_\varepsilon)H(-x))$.

\begin{lemma}
\label{l21} Assume that the flux functions  $f$ and $g$ from
\eqref{oslo3} satisfy \eqref{cond_brg}. Denote by $(v_\varepsilon)$
family of solutions to \eqref{aug2817} with $u_0\in BV(\R;[a,b])$.
Then, there exists a function $v\in L^\infty(\R)$ such that
\begin{equation}
\label{brg_5} f(v_\varepsilon)H(x)+g(v_\varepsilon)H(-x) \to
f(v)H(x)+g(v)H(-x)
\end{equation} strongly in $L^1_{loc}(\R^+\times \R)$. Moreover, the function $v$ admits left and right traces at the
interface $x=0$.
\end{lemma}
\begin{proof}
Denote
\begin{equation}
\label{brg_3} \tilde{v}_\varepsilon(t,x)=\begin{cases}
v_\varepsilon(t,x), & v_\varepsilon(t,x)\notin
\ccup\limits_{i=1}^{k_r}[a_{l_i},a_{{l_i}+1}), \ \ x> 0\\
v_\varepsilon(t,x), & v_\varepsilon(t,x)\notin
\ccup\limits_{i=1}^{k_l}[b_{l_i},b_{{l_i}+1}), \ \ x\leq
0,\\
a_{r_j}, & v_\varepsilon(t,x)\in [a_{r_j},a_{r_{j}+1}], \ \ x>0,\\
b_{l_j}, & v_\varepsilon(t,x)\in [b_{l_j},b_{l_{j}+1}], \ \ x\leq 0.
\end{cases}
\end{equation} Notice that $f(v_\varepsilon)H(x)+g(v_\varepsilon)H(-x)=f(\tilde{v}_\varepsilon)H(x)+g(\tilde{v}_\varepsilon)H(x)$
according to assumptions \eqref{cond_brg}. Then, notice that
\begin{equation}
\label{brg_6}
\begin{split}
\tilde{v}_\varepsilon\!=\!H(x)\left(\sum\limits_{i=1}^{n_r}s_{a_i,a_{i+1}}(\tilde{v}_\varepsilon)\!-\!\sum\limits_{i=2}^{n_r-1}a_i\right)
\!+\!H(-x)\left(\sum\limits_{i=1}^{n_l}s_{b_i,b_{i+1}}(\tilde{v}_\varepsilon)\!-\!\sum\limits_{i=2}^{n_l-1}b_i\right).
\end{split}
\end{equation}According to Lemma \ref{l20} and the definition of the
function $\tilde{v}_\varepsilon$, it is easy to see that
$(\tilde{v}_\varepsilon)$ is strongly precompact in
$L^1_{loc}(\R^+\times \R)$ (since this property has each of the
summands on the right-hand side of \eqref{brg_6}). Denote an
accumulation point of the family $(\tilde{v}_\varepsilon)$ by $v$.
Clearly, the function $v$ satisfies \eqref{brg_5}.

In order to prove that the function $v$ admits traces at the
interface, denote by $H(x)v^{a_i a_{i+1}}$, $i=1,\dots,n_r$, and
$H(-x)v^{b_i b_{i+1}}$, $i=1,\dots,n_l$, strong $L^1_{loc}$-limits
along subsequences of the families
$(s_{a_i,a_{i+1}}(\tilde{v}_\varepsilon))$, $i=1,\dots,n_r$, and
$(s_{b_i,b_{i+1}}(\tilde{v}_\varepsilon))$, $i=1,\dots,n_l$,
respectively. From \eqref{brg_6}, it follows:
\begin{equation}
\label{brg_6'}
\begin{split}
{v}=H(x)\left(\sum\limits_{i=1}^{n_r}v^{a_i,a_{i+1}}-\sum\limits_{i=2}^{n_r-1}a_i\right)
+H(-x)\left(\sum\limits_{i=1}^{n_l}v^{b_i,b_{i+1}}-\sum\limits_{i=2}^{n_l-1}b_i\right).
\end{split}
\end{equation} Also, notice that $H(x)v^{a_i a_{i+1}}$,
$i=1,\dots,n_r$, and $H(-x)v^{b_i b_{i+1}}$, $i=1,\dots,n_l$, are
quasi-solutions to \eqref{oslo3'}. Therefore, according to Theorem
\ref{existenceoftraces}, they admit strong traces at $x=0$. From
\eqref{brg_6'}, we see that $v$ admits strong traces as well.
\end{proof}

Now, we can prove the main theorem of the paper.

\begin{theorem}
\label{t22} Assume that the functions $\alpha$ and $\beta$ from
Definition \ref{def-adm} are such that the functions $f_\alpha$ and
$g_\beta$ satisfy the crossing conditions. Then, there exists a
unique $(\alpha,\beta)$-entropy admissible weak solution to
\eqref{oslo3}.
\end{theorem}

\begin{proof}
At the beginning, assume that  $u_0\in BV(\R;[a,b])$ and, as usual,
denote by $(v_\varepsilon)$ the family of solutions to
\eqref{aug2817}. By applying the standard procedure (see proof of
Lemma \ref{11}), it is not difficult to see that every
$v_\varepsilon$ satisfies for every $\xi\in \R$:
\begin{align}
\label{brg_11} &\pa_t \left( {\rm
sign}(v_\varepsilon-\xi)\left((\alpha(v_\varepsilon)-\alpha(\xi))H(x)+
(\beta(v_\varepsilon)-\beta(\xi))H(-x)\right)
\right)\\&+\pa_x\left((f_\alpha(v_\varepsilon)-f_\alpha(\xi))H(x)+
(g_\beta(v_\varepsilon)-g_\beta(\xi))H(-x)\right) \leq {\cal
O}_{{\cal D}'}(\varepsilon), \nonumber
\end{align} where ${\cal
O}_{{\cal D}'}(\varepsilon)$ is a family of distributions tending to
zero in the sense of distributions as $\varepsilon\to 0$. Letting
$\varepsilon\to 0$ in \eqref{brg_11} and taking Lemma \ref{l21} and
Theorem \ref{pablo} into account, we obtain in ${\cal D}'(\R^+\times
\R)$:
\begin{align}
\label{brg_12} &\pa_t \int_{\R}{\rm
sign}(\lambda-\xi)\left((\alpha(\lambda)-\alpha(\xi))H(x)+
(\beta(\lambda)-\beta(\xi))H(-x)
\right)d\nu_{t,x}(\lambda)\\&+\pa_x\left((f_\alpha(v)-f_\alpha(\xi))H(x)+
(g_\beta(v)-g_\beta(\xi))H(-x)\right) \leq 0, \nonumber
\end{align} where $\nu_{t,x}$ is a Young measure corresponding to
the sequence $(v_\varepsilon)$, and $v$ is the function satisfying
\eqref{brg_5}. The Young measure $\nu_{t,x}$ and the function $v$
(admitting strong traces at $x=0$), we shall call an
$(\alpha,\beta)$-entropy admissible measure valued solution to
\eqref{oslo3}.

Denote by $\sigma_{t,x}$ a Young measure and by $w$ a function
representing an $(\alpha,\beta)$-entropy admissible measure valued
solution to \eqref{oslo3} corresponding to initial data $v_0\in
BV(\R;[a,b])$.

Using the classical arguments by DiPerna \cite{Dpe}, we conclude
that for any test function $\varphi\in
C^1_{0}(\R^+\times(\R\backslash \{0\}))$ it holds (keep in mind that
$\alpha$ and $\beta$ are strictly increasing functions):

\begin{align}
\label{brg_13}
&\int_{\R^+\times\R}\!\int_{\R^2}\left(|\alpha(\lambda)\!-\!\alpha(\xi)|H(x)\!+\!
|\beta(\lambda)\!-\!\beta(\eta)|H(-x) \right)\pa_t\varphi
d\nu_{t,x}(\lambda)d\sigma_{t,x}(\eta)dx
dt\\&\!+\!\int_{\R^+\times\R}\!\left((f_\alpha(v)\!-\!f_\alpha(w))H(x)\!+\!
(g_\beta(v)\!-\!g_\beta(w))H(-x)\right)\pa_x\varphi dx dt \geq 0.
\nonumber
\end{align} Now, we follow \cite{kar3}. Take the function
$$
\mu_h(x)=
\begin{cases}
\frac{1}{h}(x+h), & x\in [-2h,-h]\\
1, & x\in [-h,h]\\
\frac{1}{h}(2h-x)\\
0, & |x|>2h
\end{cases},
$$ and for an arbitrary $\psi\in C^1_0(\R^+\times \R)$, put $\varphi= (1-\mu_h) \psi$ in
\eqref{brg_13}. We obtain:
\begin{align}
\label{brg_14}
&\int_{\R^+\times\R}\!\int_{\R^2}\!\left(|\alpha(\lambda)\!-\!\alpha(\xi)|H(x)\!+\!
|\beta(\lambda)\!-\!\beta(\eta)|H(-x) \right)\pa_t\psi
d\nu_{t,x}(\lambda)d\sigma_{t,x}(\eta)dx
dt\\&+\!\int_{\R^+\times\R}\!\!\left((f_\alpha(v)\!-\!f_\alpha(w))H(x)\!+\!
(g_\beta(v)\!-\!g_\beta(w))H(-x)\right)\pa_x\psi dx dt \geq
\!-\!J(h)\!+\!{\cal O}(h), \nonumber
\end{align} where $J(h)=\int_{\R^+\times \R}\left(\left((f_\alpha(v)-f_\alpha(w))H(x)+
(g_\beta(v)-g_\beta(w))H(-x)\right) \right)\mu_h'\psi dx dt$, while
${\cal O}(h)$ is the standard Landau symbol. Since $v$ and $w$ admit
strong traces at $x=0$, and since $f_\alpha$ and $g_\beta$ satisfy
the crossing conditions, as in \cite[Theorem 2.1 ]{kar3}, we
conclude that $\lim\limits_{h\to 0}J(h)\geq 0$. From here, after
letting $h\to 0$ in \eqref{brg_14}, we conclude:
\begin{align*}
&\int_{\R^+\times\R}\int_{\R^2}\left(|\alpha(\lambda)-\alpha(\eta)|H(x)+
|\beta(\lambda)-\beta(\eta)|H(-x) \right)\pa_t\psi
d\nu_{t,x}(\lambda)d\sigma_{t,x}(\eta)dx
dt\\&+\int_{\R^+\times\R}\left((f_\alpha(v)-f_\alpha(w))H(x)+
(g_\beta(v)-g_\beta(w))H(-x)\right)\pa_x\psi dx dt \geq 0, \nonumber
\end{align*} and from here, using well known procedure \cite{Kru},
we conclude that for any $T,R>0$ and appropriate $C,\bar{R}$
depending on $R$, the functions $f$, $g$, $\alpha$ and $\beta$:
\begin{align}
\label{brg_20}
\int_{0}^T\!\int_{-R}^{R}\!\int_{\R^2}\!\!\left(|\alpha(\lambda)\!-\!\alpha(\xi)|H(x)\!+\!
|\beta(\lambda)\!-\!\beta(\eta)|H(-\!x)
\right)d\nu_{t,x}(\lambda)d\sigma_{t,x}(\eta) dx dt&\\ \leq C
T\int_{-\bar{R}}^{\bar{R}}|u_0-v_0|dx&. \nonumber
\end{align} Taking $u_0=v_0$, we see from \eqref{brg_20} that for
almost every $(t,x)\in [0,T]\times \R$  the Young measures
$\nu_{t,x}$ and $\sigma_{t,x}$ are the same and they are supported
at the same point (since $\alpha$ and $\beta$ are increasing
functions). This actually means that
$\sigma_{t,x}(\lambda)=\nu_{t,x}(\lambda)=\delta(\lambda-u(t,x))$
for a function $u$, where $\delta$ is the Dirac $\delta$ function.
From Theorem \ref{pablo}, we conclude that $v_\varepsilon\to u$
strongly in $L^1_{loc}(\R^+\times \R)$ along a subsequence. The
function $u$ will obviously represent the $(\alpha,\beta)$-entropy
admissible solution to \eqref{oslo3}.

In order to prove that $u$ is a unique $(\alpha,\beta)$-entropy
admissible solution to \eqref{oslo3}, we basically need to repeat
the procedure from the first part of the proof.

Accordingly, take two $(\alpha,\beta)$-entropy admissible solutions
$u$ and $v$ to \eqref{oslo3} corresponding to initial data $u_0$ and
$v_0$, respectively. By using the same argumentation as before, we
reach to the relation analogical to \eqref{brg_14}:
\begin{align}
\label{brg_14'} &\int_{\R^+\times\R}\left(|\alpha(u)-\alpha(v)|H(x)+
|\beta(u)-\beta(v)|H(-x) \right)\pa_t\psi dx
dt\\&\!+\!\int_{\R^+\times\R}\left((f_\alpha(u)\!-\!f_\alpha(v))H(x)\!+\!
(g_\beta(u)\!-\!g_\beta(v))H(-x)\right)\pa_x\psi dx dt \!\geq\!
-\!\tilde{J}(h)\!+\!{\cal O}(h), \nonumber
\end{align} where $\tilde{J}(h)=\int_{\R^+\times\R}\left((f_\alpha(u)-f_\alpha(v))H(x)+
(g_\beta(u)-g_\beta(v))H(-x)\right) \mu_h'(x) \psi dx dt$.

Introduce the functions
\begin{align*}
&\tilde{u}(t,x)=\begin{cases} u(t,x), & u(t,x)\notin
\ccup\limits_{i=1}^{k_r}[a_{l_i},a_{{l_i}+1}), \ \ x> 0\\
u(t,x), & u(t,x)\notin
\ccup\limits_{i=1}^{k_l}[b_{l_i},b_{{l_i}+1}), \ \ x\leq
0,\\
a_{r_j}, & u(t,x)\in [a_{r_j},a_{r_{j}+1}], \ \ x>0,\\
b_{l_j}, & u(t,x)\in [b_{l_j},b_{l_{j}+1}], \ \ x\leq 0.
\end{cases},\\
&\tilde{v}(t,x)=\begin{cases} v(t,x), & v(t,x)\notin
\ccup\limits_{i=1}^{k_r}[a_{l_i},a_{{l_i}+1}), \ \ x> 0\\
v(t,x), & v(t,x)\notin
\ccup\limits_{i=1}^{k_l}[b_{l_i},b_{{l_i}+1}), \ \ x\leq
0,\\
a_{r_j}, & v(t,x)\in [a_{r_j},a_{r_{j}+1}], \ \ x>0,\\
b_{l_j}, & v(t,x)\in [b_{l_j},b_{l_{j}+1}], \ \ x\leq 0.
\end{cases}
\end{align*} Using the same arguments as in Lemma \ref{l21},
we conclude that the functions $\tilde{u}$ and $\tilde{v}$ have strong traces at the interface
$x=0$. Moreover,
$f(u)H(x)+g(u)H(-x)=f(\tilde{u})H(x)+g(\tilde{u})H(-x)$ and
$f(v)H(x)+g(v)H(-x)=f(\tilde{v})H(x)+g(\tilde{v})H(-x)$. Having this
in mind, we conclude
$$
\lim\limits_{h\to 0}
J(h)=-\int_0^T\left((f_\alpha(\tilde{u}^+)-f_\alpha(\tilde{v}^+))H(x)+
(g_\beta(\tilde{u}^-)-g_\beta(\tilde{v}^-))H(-x)\right)  \psi(t,0)
dt,
$$ where $\tilde{u}^+$ and $\tilde{u}^-$ are right and left traces
of the function $\tilde{u}$, while $\tilde{v}^+$ and $\tilde{v}^-$
are right and left traces of the function $\tilde{v}$. Now, relying
on \cite[Theorem 2.1]{kar3} again, we conclude that
$\lim\limits_{h\to 0} \tilde{J}(h) \leq 0$. From here, letting $h\to
0$ in \eqref{brg_14'}, we obtain:
\begin{align}
\label{brg_30}
 &\int_{\R^+\times\R}\left(|\alpha(u)-\alpha(v)|H(x)+ |\beta(u)-\beta(v)|H(-x)
\right)\pa_t\psi dx
dt\\&+\int_{\R^+\times\R}\left((f_\alpha(u)-f_\alpha(v))H(x)+
(g_\beta(u)-g_\beta(v))H(-x)\right)\pa_x\psi dx dt \geq 0, \nonumber
\end{align} and from here, as usual,
\begin{align*}
 &\int_{0}^T\int_{-R}^R\left(|\alpha(u)-\alpha(v)|H(x)+ |\beta(u)-\beta(v)|H(-x)
\right) dx dt\\&\leq C T\int_{-\bar{R}}^{\bar{R}}{\rm
sign}(u-v)\left((\alpha(u_0)-\alpha(v_0))H(x)+
(\beta(u_0)-\beta(v_0))H(-x) \right) dx dt. \nonumber
\end{align*}Since $\alpha$ and $\beta$ are increasing functions on the range
of $u$ and $v$, from the above we immediately obtain the $L^1_{loc}$
stability of the $(\alpha,\beta)$-entropy admissible solutions to
\eqref{oslo3}.

Now, as in the last part of the proof of Theorem \ref{th-main}, we
consider the case $u_0\notin BV(\R)$. We recall briefly the
arguments providing the statement of the theorem in this case.
First, we take a sequence $(u_{0\varepsilon})$ of the functions of
bounded variation such that $u_{0\varepsilon}\to u_0$ in
$L^1_{loc}(\R)$. Then, we take the sequence $(u_\varepsilon)$ of
$(\alpha,\beta)$-entropy admissible solutions to \eqref{oslo3} with
$u_0=u_{0\varepsilon}$. The sequence $(u_\varepsilon)$ satisfy:
$$
\int_0^T\int_{-R}^R|u_{\varepsilon_1}-u_{\varepsilon_2}|dxdt\leq
CT\int_{-\bar{R}}^{\bar{R}}|u_{0\varepsilon_1}-u_{0\varepsilon_2}|dx,
$$ where $R$ and $T$ are arbitrary positive constants, and
$C, \bar{R}$ are constants depending on $R$, the functions $f$, $g$,
$\alpha$ and $\beta$. This readily implies that the sequence
$(u_\varepsilon)$ is convergent in $L^1_{loc}(\R^+\times\R)$. Its
limit is clearly an $(\alpha, \beta)$-entropy admissible solution to
\eqref{oslo3}. Uniqueness of such $(\alpha, \beta)$-entropy
admissible solution is proved in the completely same way as when
$u_0\in BV(\R;[a,b])$.
\end{proof}

A simple corollary of Theorem \ref{t22} is the maximum principle for
an $(\alpha,\beta)$-entropy admissible solution to \eqref{oslo3}.

\begin{corollary}
\label{c23} Assume that $u$ and $v$ are two
$(\alpha,\beta)$-admissible weak solutions to \eqref{oslo3}
corresponding to the initial data $u_0\in L^1(\R;[a,b])$ and $v_0\in
L^1(\R;[a,b])$ such that $u_0(x)\leq v_0(x)$ for a.e. $x\in \R$.
Furthermore, assume that $f_\alpha$ and $g_\beta$ satisfy the
crossing conditions. Then, it holds
$$
u(t,x)\leq v(t,x) \ \ a.e. \ \ (t,x)\in \R^+\times \R.
$$
\end{corollary}
\begin{proof}
It is enough to notice that, since $|u|^+=\frac{|u|+u}{2}$, i.e.
${\rm sign}_+(u)=(|u|^+)'=\frac{{\rm sign}(u)+1}{2}$, relation
\eqref{brg_30} holds if we replace there ${\rm sign}$ by ${\rm
sign_+}$. From that relation, the standard arguments provide
$$
\int_0^T\int_{-R}^R|u(t,x)-v(t,x)|^+dx dt \leq C T
\int_{-\bar{R}}^{\bar{R}}|u_0(x)-v_0(x)|^+dx.
$$ From here, the statement of the corollary immediately follows.
\end{proof}

Now, we shall prove that we can always find $\alpha$ and $\beta$ so
that there exists a unique $(\alpha,\beta)$-entropy admissible
solutions to \eqref{oslo3}.

\begin{figure}[htp]
\begin{center}
  \includegraphics[width=4in]{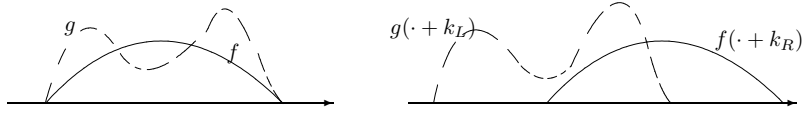}\\
  \caption{Functions $f$ (normal line) and $g$
(dashed line) on the left plot do not satisfy the crossing
condition. On the other hand, for appropriate $k_L>k_R$, the
functions $f(\cdot+k_R)$ and $g(\cdot+k_L)$ on the right plot
satisfy the crossing conditions.}
 \label{3}
\end{center}
\end{figure}


\begin{theorem}
Denote by $\chi_{[a,b]}$ the characteristic function of the interval
$[a,b]$. For the functions $\alpha(u)=\alpha_T(u)=u+k_R$ and
$\beta(u)=\beta_T(u)=u+k_L$ such that the functions
$f^c_{\alpha}=(f\chi_{[a,b]})\circ \alpha$ and
$g^c_\beta=(g\chi_{[a,b]})\circ \beta$ satisfy the crossing
conditions, there exists a unique $(\alpha,\beta)$-entropy
admissible solution to \eqref{oslo3}.
\end{theorem}

\begin{proof}
First, notice that it is always possible to find constants $k_R$ and
$k_L$ such that the translation functions $\alpha_T(u)=u+k_R$ and
$\beta_T(u)=u+k_L$ make $f^c_{\alpha_T}$ and $g^c_{\beta_T}$ to
satisfy the crossing conditions (see Figure \ref{3}). Furthermore,
the constants $a$ and $b$ represent $(\alpha_T,\beta_T)$-entropy
admissible solutions to
\begin{equation}
\label{oslo3''}
\begin{cases}
\pa_t u+\pa_x\left(H(x)(f\chi_{[a,b]})(u)+H(-x)(g\chi_{[a,b]})(u) \right)=0, & (t,x)\in  \R^+\times \R\\
u|_{t=0}=u_0(x)\in L^\infty(\R), & x\in \R.
\end{cases}
\end{equation}

Indeed, denoting $k(x)=\begin{cases} k_L, &x\leq 0\\
k_R, &x>0
\end{cases}$, according to Definition \ref{def-adm}, we see that we
need to check whether the function $v(t,x)=a-k(x)$ satisfies
\eqref{sep827}. After substituting it there, we see that we need to
check whether (see also \cite[Remark 2]{NHM_mit})

\begin{align}
\label{brg_40} \Big({\rm
sgn}(a-k_L-\xi)(g\chi_{[a,b]})(\xi+k_L)&-{\rm
sgn}(a-k_R-\xi)(f\chi_{[a,b]})(\xi+k_L)
\\&-|(f\chi_{[a,b]})(\xi+k_R)-(f\chi_{[a,b]})(\xi+k_L)|\Big)\delta(x)\leq 0,
\nonumber
\end{align} for every $\xi\in \R$. Clearly, if $\xi\in \R$ is such
that $\min\{a-k_R-\xi,a-k_L-\xi \}\! \geq \! 0$ or $\max
\{a-k_R-\xi,a-k_L-\xi \}\leq 0$, then \eqref{brg_40} holds with the
equality sign. Otherwise, it must hold $a-k_R-\xi \leq 0 \leq
a-k_L-\xi$ (see Figure \ref{3}). However, if this is a case, then
$\xi+k_L\leq a$. This implies $(g\chi_{[a,b]})(\xi+k_L)=0$ from
where \eqref{brg_40} easily follows. Similarly, we prove that
$u(t,x)\equiv b$ represents an $(\alpha_T,\beta_T)$-entropy
admissible solution to \eqref{oslo3}.

From here, using Corollary \ref{c23}, we conclude that, for the
$\alpha_T$ and $\beta_T$ chosen above (Figure \ref{3}), the
$(\alpha_T,\beta_T)$-entropy admissible solutions to \eqref{oslo3},
say $u$, such that $a\leq u_0 \leq b$, must satisfy $a \leq u(t,x)
\leq b$ for a.e. $(t,x)\in \R^+\times \R$. This actually means that
the $(\alpha_T,\beta_T)$-entropy admissible solution to
\eqref{oslo3''} is, at the same time, $(\alpha_T,\beta_T)$-entropy
admissible solution to \eqref{oslo3} (since on the range of the
solution it holds $f\chi_{[a,b]}\equiv f$ and $g\chi_{[a,b]}\equiv
g$). Since $f^c_{\alpha_T}$ and $g^c_{\beta_T}$ satisfy the crossing
conditions, according to Theorem \ref{t22}, we conclude that the
$(\alpha_T,\beta_T)$-entropy admissible solution to \eqref{oslo3''}
is unique making it a unique solution to \eqref{oslo3}. \end{proof}

{\bf Acknowledgement:} The work is initiated and, in the main part,
written while the author was postdoc at NTNU. It is supported in
part by the Research Council of Norway. Main ideas of the paper were
presented at the conference Multiscale Problems in Science and
Technology, Dubrovnik 2007.

\end{document}